\documentclass{amsart}

\usepackage{graphicx}
\usepackage{amsthm}
\usepackage{xcolor}
\newtheorem{theorem}{Theorem}[section]
\newtheorem*{theorem*}{Theorem}
\newtheorem{prop}[theorem]{Proposition}
\newtheorem{lemma}[theorem]{Lemma}

\newtheorem{cor}[theorem]{Corollary}
\usepackage{tikz}
\usepackage{subfigure}
\usepackage{varwidth}
\usetikzlibrary{positioning}
\theoremstyle{definition}
\newtheorem{definition}[theorem]{Definition}

\usepackage{verbatim}
\theoremstyle{remark}

\usepackage[justification=centering]{subfig}
\numberwithin{equation}{section}

\usetikzlibrary{cd}
\usetikzlibrary{decorations.pathmorphing}
\usepackage{tikz-qtree}
\DeclareMathOperator{\F}{\mathcal{F}}
\DeclareMathOperator{\B}{\mathcal{B}}

\begin{document}
\tikzset{
        ribbon/.style={
            preaction={
                draw,
                line width=0.2cm,
                black!30!#1
            },
            line width=0.15cm,
            #1
        },
        ribbon/.default=gray
    }
    
\tikzset{
        gribbon/.style={
                    preaction={
                draw,
                line width=0.2cm,
                green!#1
            },
            line width=0.15cm,
            #1
        },
        ribbon/.default=gray
    }
    
\tikzset{
        bribbon/.style={
            preaction={
                draw,
                line width=0.2cm,
                blue!#1
            },
            line width=0.15cm,
            #1
        },
        ribbon/.default=gray
    }
\title{Interlacement and Activities in Delta-matroids}

%    Information for first author
\author{Ada Morse}
%    Address of record for the research reported here
\address{Department of Mathematics and Statistics, University of Vermont, Burlington, Vermont 05401}
%    Current address
%\curraddr{Department of Mathematics and Statistics,
%Case Western Reserve University, Cleveland, Ohio 43403}
%\email{xyz@math.university.edu}
%    \thanks will become a 1st page footnote.
%\thanks{The first author was supported in part by NSF Grant \#000000.}

%    Information for second author

%    General info

\date{\today}

%\dedicatory{This paper is dedicated to our advisors.}


\begin{abstract} We generalize theories of graph, matroid, and ribbon-graph activities to delta-matroids. As a result, we obtain an activities based feasible-set expansion for a transition polynomial of delta-matroids defined by Brijder and Hoogeboom. This result yields feasible-set expansions for the two-variable Bollob\'{a}s-Riordan and interlace polynomials of a delta-matroid. In the former case, the expansion obtained directly generalizes the activities expansions of the Tutte polynomial of graphs and matroids.
\end{abstract}

\maketitle

\section{Introduction}

Delta-matroids are a generalization of matroids that have been the subject of increased interest recently in part due to the rediscovery \cite{2016arXiv160201306C,CMNR2014} of a connection (originally due to Bouchet \cite{DBLP:journals/dm/Bouchet89}) between delta-matroids and embedded graphs that generalizes the classical connection between matroids and abstract graphs. Delta-matroids also arise, independently, in the study of skew-symmetric matrices, and have a direct connection to abstract graphs by way of the adjacency matrix (a connection which does not generalize the classical connection between matroids and abstract graphs.)

In the context of delta-matroids arising from the adjacency matrix, Brijder and Hoogeboom have defined a transition polynomial $Q_{(w,x,0)}(D;y)$ of delta-matroids satisfying a deletion-contraction property reminiscent of the Tutte polynomial of a matroid \cite{BH2014}. In this paper, we show that this transition polynomial also satisfies a delta-matroid analog of the activities expansion of the Tutte polynomial. Recall that the activities expansion of the Tutte polynomial is obtained by associating, to each basis $B$ of a matroid $M = (E,\mathcal{B})$, a set of ``internally active'' edges and a set of ``externally active'' edges according to some arbitrary total order $<$ on the groundset of $E$ (see Section 3 for details.) The analog of bases in delta-matroids are called feasible sets, and, by generalizing the spanning quasi-tree activities of \cite{De,CKS2011,VT2011}, we show that incorporating orientability of points in the definition of activity yields a feasible-set expansion for the transition polynomial:

\begin{theorem*} \label{thm:transfeas} Let $D = (E,\mathcal{F})$ be a delta-matroid. Let $<$ be any total order on $E$. Let $F \in \mathcal{F}$. Let $i(F)$ be the number of internal, active, and orientable points with respect to $F$ and let $j(F)$ be the number of external, active, and orientable points with respect to $F$. Then
\begin{equation*}
    Q_{(w,x,0)}(D;y) = \sum_{F \in \mathcal{F}} w^{|E| - |F|} x^{|F|} (1 + (w/x)y)^{i(F)} (1 + (x/w)y)^{j(F)}
\end{equation*}
\end{theorem*}

As applications of this result, we also obtain feasible-set expansions for the two-variable Bollob\'{a}s-Riordan and interlace polynomials. The feasible-set expansion of the Bollob\'{a}s-Riordan polynomial is of particular interest, as it directly  generalizes the basis expansion of the Tutte polynomial. We conclude by discussing open questions regarding these feasible-set expansions as well as activities expansions more generally.
\section{Preliminaries}

\subsection{Delta-matroids} Delta-matroids are a generalization of matroids that were introduced independently by many authors in the late 1980s \cite{B87greedy,DH86,CK88}.

\begin{definition} A \emph{delta-matroid} is a pair $D = (E,\F)$ where $E$ is a finite set, called the \emph{ground set}, and $\F$ is a nonempty collection of subsets of $E$ called \emph{feasible sets} satisfying the following \emph{symmetric exchange axiom}:
\begin{center} for all $X,Y \in \F$ and $a \in X \Delta Y$,\\ there exists $b \in X \Delta Y$, not necessarily distinct from $a$, such that $X \Delta \{a,b\} \in \F$
\end{center}
where $X \Delta Y = (X \cup Y) \setminus (X \cap Y)$ is the usual symmetric difference of sets. When necessary to specify the underlying delta-matroid, we will denote $E$ by $E(D)$ and $\F$ by $\F(D)$.
\end{definition}

Note that there are essentially three cases of the symmetric exchange axiom: either $a \in X \setminus Y$ and $b \in Y \setminus X$ or vice versa (this is the usual basis exchange axiom for matroids); both $a$ and $b$ are in $X \setminus Y$; or both $a$ and $b$ are in $Y \setminus X$. If all the feasible sets in $\F$ have the same cardinality, the latter two cases cannot occur, and so the first holds for all feasible sets $X$ and $Y$. That is, a matroid is precisely a delta-matroid all of whose feasible sets have the same size.

We will need the following basic definitions. Let $D = (E,\F)$ be a delta-matroid.  A point $a \in E$ contained in every feasible set of $D$ is said to be a \emph{coloop}, while a point $a \in E$ contained in no feasible set of $D$ is said to be a \emph{loop}. A point is \emph{singular} if it is either a loop or a coloop, and \emph{nonsingular} otherwise. Let $\mathcal{F}_{\max}$ be the collection of maximum cardinality feasible sets of $D$ and $\mathcal{F}_{\min}$ the collection of minimum cardinality feasible sets of $D$. It can be shown that $D_{\max} = (E,\mathcal{F}_{\max})$ and $D_{\min} = (E,\mathcal{F}_{\min})$ are indeed matroids. Let $r_{\max}$ be the matroid rank function of $D_{\max}$ and let $r_{\min}$ be the matroid rank function of $D_{\min}$.

We will require three operations on delta-matroids: deletion, contraction, and twisting.

\begin{definition} (Deletion and contraction) Let $a \in E$. If $a$ is not a coloop, define \emph{$D$ delete $E$} to be the set system $D \setminus a = (E - a, \{F : F \in \F, a \not \in F\})$. If $a$ is not a loop, define \emph{$D$ contract $E$} to be the set system $D/a = (E - a, \{F - a: F \in \F, a \in F\})$. If $a$ is a loop, define $D/a$ to be $D\setminus a$. If $a$ is a coloop, define $D\setminus a$ to be $D/a$. It can be shown that the order in which deletions and contractions are performed does not matter. If $A \subseteq E$ then the \emph{restriction} of $D$ to $A$ is the delta-matroid $D|A = D \setminus (E \setminus A)$.
\end{definition}

Observe that if $\emptyset \in \F$, then there are no coloops in $E$, and so, for any $A \subseteq E$, $\F(D|A) = \{F \in \F: F \subseteq A\}$.

\begin{definition} (Twist) For $X \subseteq E$, define the \emph{twist} $D*X := (E, \{X \Delta F: F \in \F\})$. 
\end{definition}
Using the identity $(A \Delta C) \Delta (B \Delta C) = A \Delta B$, it is straightforward to show that the twist of a delta-matroid is a delta-matroid. Note that the twist of a matroid is not necessarily a matroid (although it is, of course, a delta-matroid). Also, observe that if $F$ is itself a feasible set, then $\emptyset \in \mathcal{F}(D*F)$.

We will also need the following definition of connectivity in delta-matroids, which generalizes the standard definitions of connectivity for matroids.

\begin{definition} (Connected/disconnected) We say a delta-matroid $D = (E,\F)$ is \emph{disconnected} if there exist delta-matroids $D_1 = (E_1,\F_1)$ and $D_2 = (E_2,\F_2)$ such that $E_1 \cap E_2 = \emptyset$ and $(E_1 \cup E_2, \{F_1 \cup F_2: F_1 \in \F_1, F_2 \in \F_2\}) = (E,\F)$. In this case we write $D = D_1 \oplus D_2$. If a delta-matroid is not disconnected, we say it is \emph{connected}.
\end{definition}

Given a graph $G$, we can construct a delta-matroid from its adjacency matrix as follows.

\begin{definition} (adjacency delta-matroid) Let $G$ be a graph, allowing single loops but not multiple loops or edges. Let $A$ be the $V(G) \times V(G)$ adjacency matrix of $G$, considered over the field $GF(2)$. A set $X \subseteq V(G)$ is \emph{feasible} if the principle submatrix $A[X]$ is invertible. By convention, $\emptyset$ is always feasible. The \emph{adjacency delta-matroid} of $G$ is the set system $\mathcal{A}_G = (V(G), \{ X \subseteq V(G) : X \text{ is feasible.}\}$.
\end{definition}

Note that, since $\emptyset$ is always feasible in $\mathcal{A}_G$, the adjacency delta-matroid of a graph is almost never the graphic matroid of a graph. We refer the reader to e.g. \cite{Bou87b,BD91} for background on adjacency delta-matroids and representability of delta-matroids more generally.

\subsection{Interlacement and ribbon graphs} \label{sec:ribbon} Interlacement in ribbon graphs is a generalization of interlacement in double-occurrence words, which has been studied extensively in connection to a question of Gauss regarding which double-occurrence words can be represented by Eulerian circuits in plane 4-regular graphs \cite{B87circle,B87circle2,B94,F97,RR76,R99}. We will require the following definitions from ribbon graph theory, for which we follow closely the development of \cite{EMM2013}. 

A \emph{ribbon graph} $G = (V(G),E(G))$ is a surface with boundary presented as the union of two sets of discs, a set $V(G)$ of \emph{vertices} and a set $E(G)$ of \emph{edges}, satisfying the following conditions:
\begin{enumerate}
    \item The vertices and edges intersect in disjoint line segments.
    \item Each such line segment lies on the boundary of precisely one vertex and precisely one edge.
    \item Every edge contains exactly two such line segments.
\end{enumerate}

Throughout the paper, we will assume all ribbon graphs are connected. Note that ribbon graphs are equivalent to cellularly embedded graphs, where a \emph{cellular embedding} of a graph $G$ on a closed compact surface $\Sigma$ is a drawing of $G$ on $\Sigma$ such that edges intersect only at their endpoints and each component of $\Sigma - G$ is homeomorphic to a disc. Two cellularly embedded graphs in the same surface are considered \emph{equivalent} if there is a homeomorphism of the surface taking one to the other. A cellularly embedded graph can be obtained from a ribbon graph by gluing discs to the holes in the ribbon graph and retracting the ribbon graph. Ribbon graphs are considered \emph{equivalent} if their associated cellularly embedded graphs are equivalent.

A \emph{spanning quasi-tree} of a ribbon graph $G$ is a spanning subgraph of $G$ having exactly one boundary component. 

\begin{figure}[h]
    \centering
    \includegraphics[scale=.7]{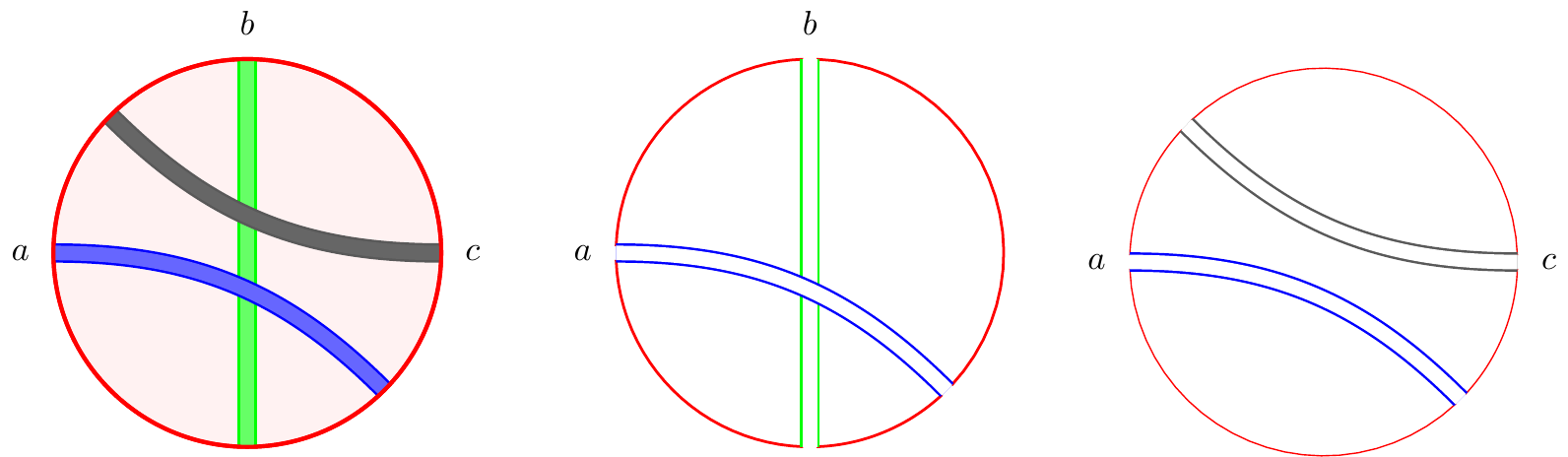}
    \caption{Left to right: a single-vertex ribbon graph with ribbon edge loops drawn arching above the vertex disc; the single boundary component demonstrating that the edges $a$ and $b$ form a quasi-tree; the three boundary components demonstrating that the edges $a$ and $c$ do not.}
    \label{fig:rd}
\end{figure}

Let $G$ be a ribbon graph and let $A \subseteq E(G)$. The \emph{partial dual} $G^{A}$ of $G$ is formed as follows: regard the boundary components of the induced ribbon subgraph $(V(G),A)$ of $G$ as curves on the surface of $G$. Glue a disc to $G$ along each connected component of this curve and remove the interior of all vertices of $G$. In the case that $Q$ is a spanning quasi-tree of $G$, we will write $G^Q$ for $G^{E(Q)}$. If $e \in E(G)$, then $G \setminus e$ is the ribbon graph $(V(G),E(G) - e)$. We define $G/e$ to be the ribbon graph $G^{e}\setminus e$. We remark the order in which edges are deleted or contracted does not matter, and so if $A$ and $B$ are disjoint subsets of $E(G)$, we write $G / A \setminus B$ for the ribbon graph obtained by contracting the edges in $A$ in any order and deleting the edges in $B$ in any order. The following well-known pronerty of ribbon graphs (see e.g. \cite{VT2011}) will be useful in defining interlacement.

\begin{prop} Let $G$ be a ribbon graph with spanning quasi-tree $Q$. Then $G^{Q}$ is a single-vertex ribbon graph.
\end{prop}

We can now define interlacement with respect to quasi-trees as in \cite{VT2011}.

\begin{definition} \label{def:ribbon_int} Let $G$ be a ribbon graph with spanning quasi-tree $Q$. Let $e,f \in E(G)$. We say $e$ and $f$ are \emph{interlaced with respect to $Q$} or \emph{$Q$-interlaced} if they are met in the order $e...f...e...f...$ while traversing the boundary of the single vertex of $G^Q$.
\end{definition}

Recall that an edge $e$ in a single-vertex ribbon graph $G$ is called \emph{nonorientable} if the surface formed by $e$ together with the vertex is homeomorphic to a M\"{obius} band \cite{EMM2013}. In the case that $G$ is an arbitrary ribbon graph with spanning quasi-tree $Q$, we will say that $e \in E(G)$ is \emph{$Q$-nonorientable} if $e$ is a nonorientable loop in $G^Q$. In addition, we will say that $e,f \in E(G)$ are \emph{$Q$-paired} if $\{e,f\}$ is a quasi-tree in $G^Q$ and  \emph{$Q$-separated} otherwise. Then the following theorem is straightforward to verify (see Figure \ref{fig:ribbonchar} for one case.)

\begin{theorem} \label{thm:ribbonchar} Let $G$ be a ribbon graph with spanning quasi-tree $Q$. Let $e,f \in E(G)$. Then $e$ and $f$ are $Q$-interlaced if and only if one of the following holds: (1) at most one of $e$ or $f$ is $Q$-nonorientable and $\{e,f\}$ is $Q$-paired, or (2) both $e$ and $f$ are $Q$-nonorientable and $\{e,f\}$ is $Q$-separated.

\begin{figure}
    \centering
   \subfigure[interlaced loops]{%
\includegraphics[height=2in]{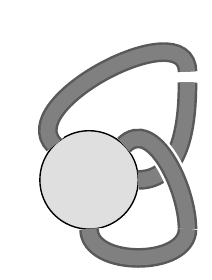}}%
\qquad
\subfigure[uninterlaced loops]{%
\includegraphics[height=2in]{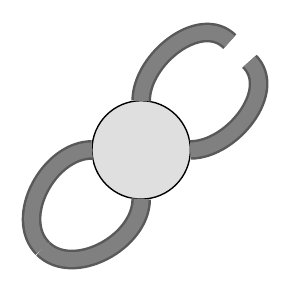}}%

    \caption{A figure to accompany the case of Theorem \ref{thm:ribbonchar} where at most one of the two loops is $Q$-nonorientable. The segments of the upper loop in each subfigure can be joined with a twist or without a twist. In either case, tracing boundary components shows that the interlaced loops form a quasi-tree while the uninterlaced loops do not. Hence, the interlaced loops are $Q$-paired while the uninterlaced loops are not.}
    \label{fig:ribbonchar}
\end{figure}
\end{theorem}

 In the next section, we will use this characterization of interlacement to motivate a definition of interlacement for delta-matroids.

\subsection{Delta-matroids and ribbon graphs}

We will be exploiting a connection between delta-matroids and ribbon graphs, initially observed by Bouchet \cite{DBLP:journals/dm/Bouchet89} and recently developed further by Chun et. al. \cite{CMNR2014}. 

\begin{theorem} \cite{CMNR2014} Let $G$ be a ribbon graph. Let $\F$ be the set of all spanning quasi-trees of $G$. Then $D(G) := (E(G), \F)$ is a delta-matroid, called the \emph{graphic delta-matroid} of $G$.
\end{theorem}

Note that if $G$ is a plane ribbon graph, the spanning quasi-trees of $G$ are precisely the spanning trees of $G$, and so $D(G)$ is the usual graphic matroid of $G$. Chun et al. have also shown the following equivalence between twisting and partial duals.

\begin{theorem} \cite{CMNR2014} \label{thm:twistdual} Let $G$ be a ribbon graph. Then $D(G^A) = D(G) * A$ for any $A \subseteq E(G)$.
\end{theorem}

Note that if $F$ is a feasible set in $D(G)$ then our previous observation that $\emptyset \in \mathcal{F}(D(G) * F)$ is equivalent to the statement that $G^F$ is a single vertex ribbon graph.

\section{Fundamental graphs and activities}

In this section we lift the idea of interlacement to general delta-matroids, laying a foundation for a definition of activities that will lead in Sections 4 and 5 to the desired activities expansions of the transition, interlace, and Bollob\'{a}s-Riordan polynomials.

An essential object in delta-matroid theory is the fundamental graph of a delta-matroid with respect to a feasible set. While these fundamental graphs arise naturally when working with delta-matroids in the abstract (see \cite{B01} for even delta-matroids and \cite{Geelen} for general delta-matroids), we will derive the definition by analogy to interlacement in ribbon graphs. Note that if $G$ is a ribbon graph with spanning quasi-tree $Q$, an edge $e$ is $Q$-nonorientable if and only if $\{e\}$ is a feasible set in $G^Q$. This allows us to generalize notions of orientability and pairing with respect to quasi-trees to arbitrary delta-matroids.

\begin{definition} \label{def:delta_int} Let $D = (E,\F)$ be a delta-matroid and let $F \in \F$. A point $a \in E$ is \emph{$F$-nonorientable} if $\{a\}$ is feasible in $D*F$, and \emph{$F$-orientable} otherwise. A pair $\{a,b\} \subseteq E$ is \emph{$F$-paired} if $\{a,b\}$ is feasible in $D*F$, and is \emph{$F$-separated} otherwise. 
\end{definition}

Definition \ref{def:delta_int} streamlines, for our purposes, the language of \cite{CMNR2014} which uses ribbon loops as follows. A point $a \in E$ is a \emph{ribbon loop} in $D$ if $a$ is not contained in any feasible set in $\F_{\min}$. If $a$ is a ribbon loop in $D$, $a$ is \emph{nonorientable} if $a$ is also a ribbon loop in $D*a$, and \emph{orientable} otherwise. The following proposition gives the correspondence between $F$-orientability and ribbon loop orientability.

\begin{prop} \label{nonorientable} Let $D=(E,\F)$ be a delta-matroid and let $F \in \F$. Let $a \in E$. Then $a$ is a nonorientable ribbon loop in $D*F$ if and only if $a$ is $F$-nonorientable.
\end{prop}

\begin{proof} Suppose $a$ is a nonorientable ribbon loop in $D*F$. Then $a$ is a ribbon loop in $(D*F)*a$. Since $\emptyset$ is feasible in $D*F$, $\{a\} = \emptyset \Delta \{a\}$ is feasible in $(D*F)*a$. Since $a$ is a ribbon loop in $(D*F)*a$, $\{a\}$ is not a minimum cardinality feasible set of $(D*F)*a$, and thus $\emptyset$ must be feasible in $(D*F)*a$. It follows that there exists $F$ feasible in $D*F$ such that $F \Delta \{a\} = \emptyset$. The only possibility is $F = \{a\}$. Thus, $a$ is $F$-nonorientable. On the other hand, suppose $a$ is $F$-nonorientable. Then $\{a\} \Delta \{a\} = \emptyset$ is feasible in $(D*F)*a$. Thus, $a$ is a ribbon loop in $(D*F)*a$, i.e. $a$ is a nonorientable ribbon loop in $D*F$.
\end{proof}

Using the language of $F$-nonorientable points and $F$-pairs, we define interlacement in delta-matroids analogously to the characterization for ribbon graphs provided in Theorem \ref{thm:ribbonchar}. 

\begin{definition} \label{thm:pairs} Let $D = (E,\F)$ be a delta-matroid, let $F \in \F$, and let $a,b \in E$. We say $a$ and $b$ are \emph{$F$-interlaced} if: (1) at most one of $a$ or $b$ is $F$-nonorientable and $\{a,b\}$ is $F$-paired, or (2) both $a$ and $b$ are $F$-nonorientable and $\{a,b\}$ is $F$-separated. The set of points $F$-interlaced with $a$ is $I(F;a)$.
\end{definition}

Interlacement in delta-matroids can also be defined in terms of delta-matroid connectivity.

\begin{theorem} Let $D = (E,\mathcal{F})$ be a delta-matroid. Let $F \in \mathcal{F}$ and let $a,b \in E$. Then $a$ and $b$ are $F$-interlaced if and only if $(D*F)|\{a,b\}$ is connected and nontrivial (i.e. has a feasible set other than $\emptyset$).
\end{theorem}

\begin{proof} Suppose $D' := (D*F)|\{a,b\}$ is connected and nontrivial. Recall that since $F \in \mathcal{F}$, $\emptyset \in \mathcal{F}(D')$. We consider two cases. Suppose $\{a,b\} \in \mathcal{F}(D')$. Then at most one of $\{a\}$ or $\{b\}$ can be in $\mathcal{F}(D')$, since
\begin{equation*}
    (\{a,b\}, \{ \{a,b\}, \{a\}, \{b\}, \emptyset\}) = (\{a\}, \{ \{a\}, \emptyset\}) \oplus (\{b\}, \{ \{b\}, \emptyset \}).
\end{equation*}
Now suppose $\{a,b\} \not \in \mathcal{F}(D').$ Then since $D'$ is nontrivial,
\begin{equation*}
    (\{a,b\}, \{ \{a\}, \emptyset \}) = ( \{a\}, \{ \{a\}, \emptyset \}) \oplus (\{b\}, \{ \emptyset \}), \text{ and }
\end{equation*}
\begin{equation*}
    (\{a,b\}, \{ \{b\}, \emptyset \}) = ( \{a\}, \{ \emptyset \}) \oplus (\{b\}, \{ \{b\}, \emptyset \}),
\end{equation*}
both $\{a\}$ and $\{b\}$ are feasible in $D'$. Thus, $D'$ is one of the following delta-matroids:
\begin{enumerate}
    \item $( \{a,b\}, \{ \{a,b\}, \emptyset\})$,
    \item $( \{a,b\}, \{\{a,b\},\{a\}, \emptyset \})$,
    \item $( \{a,b\}, \{ \{a,b\},\{b\}, \emptyset \})$, or
    \item $( \{a,b\}, \{ \{a\}, \{b\}, \emptyset \})$.
\end{enumerate}

Therefore, $a$ and $b$ are $F$-interlaced. 

For the reverse implication, note that if $a$ and $b$ are $F$-interlaced, $(D*F)|\{a,b\}$ is one of the four delta-matroids above. It is  not difficult to show that each of these delta-matroids is connected and nontrivial, completing the proof.
\end{proof}

We can now give the definition of the fundamental graphs of a delta-matroid from \cite{B01,Geelen} in the language of interlacement. 

\begin{definition} Let $D = (E,\F)$ be a delta-matroid. Let $F \in \F$. The \emph{fundamental graph of $D$ with respect to $F$}, denoted $\Gamma(D,F)$, is the graph with vertex set $E$ and two vertices adjacent when they are $F$-interlaced. In this setting, $I(F;a)$ is the open neighborhood of $a$ in $\Gamma(D,F)$.
\end{definition}

\subsection{Feasible-set activities.} Motivated by the ribbon-graph activities of \cite{CKS2011,VT2011,B2012,De}, we use interlacement in delta-matroids to generalize matroid basis activities to delta-matroid feasible-set activites.

\begin{definition}
Let $D=(E,\F)$ be a delta-matroid and let $F \in \F$. Let $<$ be a total order on $E$.  A point $a \in E$ is \emph{active with respect to $F$} if it is the lowest-ordered point in $I(F;a)$, i.e. $a$ is active if $a$ is not $F$-interlaced with any lower-ordered points. A point that is not active is said to be \emph{inactive}. We say $a$ is \emph{internal} if $a \in F$ and \emph{external} otherwise. 
\end{definition}

It is not necessarily obvious that, for a matroid, this definition corresponds with the usual definition of matroid activity, which we now recall. Let $M = (E,\B)$ be a matroid described by its bases with $<$ a total order on $E$. Let $B$ be a basis of $M$ and let $a \in E$. Suppose $a \not \in B$. Then there is a unique circuit $C(B;a)$ in $B \cup a$ called a \emph{fundamental circuit}. Then recall that $a$ is said to be \emph{externally active} if it is the least element of the fundamental circuit $C(B;a)$, and \emph{externally inactive} otherwise. Suppose $a \in B$. Then $a$ is \emph{internally active} if $a$ is \emph{externally active} with respect to $E \setminus B$ in $M*E$ and \emph{internally inactive} otherwise.

\begin{prop} \label{cor:circuitinter} Let $M = (E,\B)$ be a matroid. Let $B \in \B$ and let $a \not \in B$. Then $I(B;a) = C(B;a)$.
\end{prop}

\begin{proof} Since $M$ is a matroid, all feasible sets have the same size, and so $a$ is $B$-orientable. Thus, $I(B;a)$ is precisely the set of points with which $a$ is $B$-paired. Suppose $b$ is $B$-paired with $a$. Then there exists $B' \in \B$ such that $\{a,b\} = B \Delta B'$.  Again, since $M$ is a matroid, $b \in B$ and $B \setminus \{b\} \subseteq B'$. Suppose $b \not \in C(B;a)$. Then $C(B;a) \subseteq (B \setminus \{b\}) \cup \{a\} = B'$. But this latter set is independent, a contradiction. Thus, $b \in C(B;a)$.

Now suppose $b \in C(B;a)$. Since $C(B;a)$ is the unique circuit in $B \cup \{a\}$, $B' = (B \setminus \{b\}) \cup \{a\}$ is a basis for $M$ for which $B \Delta B' = \{a,b\}$. Thus, $b$ is $B$-paired with $a$, which, since $M$ is a matroid, implies that $a$ and $b$ are $B$-interlaced.
\end{proof}

We conclude in Theorem \ref{thm:mat-delt} that the usual definition of activity for matroids corresponds with the delta-matroidal definition.

\begin{theorem} \label{thm:mat-delt} Let $M = (E,\B)$ be a matroid and let $<$ be a total order on $E$. Let $B \in \B$. Let $b \in E$. Then $b$ is externally active (thinking of $M$ as a matroid) if and only if $b$ is external and active (thinking of $M$ as a delta-matroid.) Dually, $b$ is internally active if and only if $b$ is internal and active.
\end{theorem}

\begin{proof} Suppose $b$ is externally active. Then $b \not \in B$ and $b$ is the lowest-ordered element of $C(B;b)$. Thus, $b \not \in B$ and $b$ is the lowest-ordered element of $I(B;b)$. Hence, $b$ is external and active. The reverse direction follows similarly.
\end{proof}

\section{Feasible set expansion of the transition polynomial}

In this section we use the delta-matroid activities defined above to give feasible-set expansions for delta-matroid polynomials. We will begin by computing a general feasible-set expansion for a family of transition polynomials defined by Brijder and Hoogeboom in \cite{BH2014}.

\begin{definition} (transition polynomial) \cite{BH2014} Let $D = (E,\mathcal{F})$ be a delta-matroid. An \emph{ordered 3-partition of $E$} is an ordered triple $(E_1,E_2,E_3)$ of subsets of $E$ such that $E_i \cap E_j = \emptyset$ for $i \neq j$ and $E = \bigcup E_i$. Let $P_3$ be the set of all ordered 3-partitions of $E$. The \emph{transition polynomial} $Q_{(w,x,z)}(D;y)$ is
\begin{equation*}
    Q_{(w,x,z)}(D;y) = \sum_{(A,B,C) \in P_3} w^{|A|} x^{|B|} z^{|C|} y^{r_{\min}(D * B)}.
\end{equation*}
\end{definition}

The transition polynomials $Q_{(w,x,0)}(D;y)$ are of particular interest, as they include several well-studied delta-matroid polynomials (such as the interlace polynomials and the Bollob\'{a}s-Riordan polynomial). Moreover, Brijder and Hoogeboom have shown that $Q_{(w,x,0)}(D;y)$ satisfies the following Tutte-like deletion-contraction recurrence, recalling that a point $a \in E$ is nonsingular if $a$ is neither a loop nor a coloop.

\begin{theorem} \cite{BH2014} \label{thm:transdc} Let $D = (E,\mathcal{F})$ be a delta-matroid, and let $a \in E$ be nonsingular. Then,
\begin{equation}
    Q_{(w,x,0)}(D;y) = w Q_{(w,x,0)}(D \setminus a;y) + x Q_{(w,x,0)}( D / a;y).
\end{equation}
If every element of $D$ is singular and $D$ has $c$ coloops and $l$ loops, then
\begin{equation}
    Q_{(w,x,0)}(D;y) = (x + wy)^c (w + xy)^l.
\end{equation}
\end{theorem}

To obtain a feasible-set expansion of $Q_{(w,x,0)}(D;y)$, we begin by computing the polynomial using a tree of minors of $D$ analogous to the deletion-contraction computation tree for the Tutte polynomial. Note that the tree constructed in Definition \ref{def:tree} generalizes the tree of partial resolutions considered in \cite{CKS2011,VT2011,B2012,De} for computing spanning quasi-tree expansions of ribbon-graph polynomials. We define it in terms of deletions and contractions instead of partial resolutions to make explicit the connections to the theory of deletion-contraction recurrences in matroids and graph polynomials.

\begin{definition} \label{def:tree} Let $D = (E,\mathcal{F})$ be a delta-matroid and let $<$ be a total order on $E$. We inductively construct a binary tree $\mathcal{T}(D)$, whose nodes are certain minors of $D$. The root of $\mathcal{T}(D)$ is $D$. Suppose a node $D'$ has been added to $\mathcal{T}(D)$. If every point of $D'$ is singular, then $D'$ is given no children (i.e. $D'$ will be a leaf of $\mathcal{T}(D))$. Otherwise, let $a \in E(D')$ be the highest ordered nonsingular point of $D'$. The two children of $D'$ will be $D' \setminus a$ and $D' / a$. 
\end{definition}

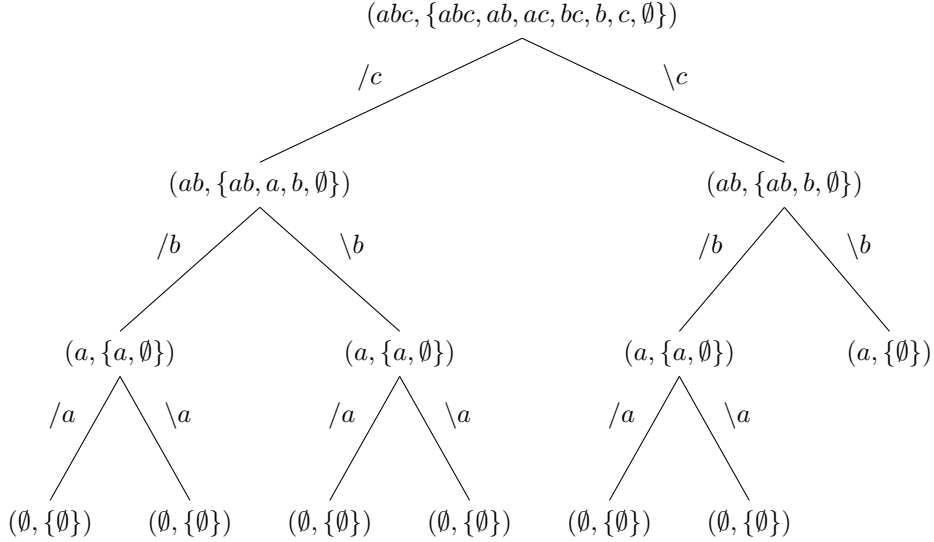
\begin{figure}
\centering
\begin{tikzpicture}
\tikzset{sibling distance=14pt,level distance=64pt}
\Tree [.{$(abc,\{abc,ab,ac,bc,b,c,\emptyset\})$} \edge node[auto=right]{$/c$}; [.{$(ab,\{ab,a,b,\emptyset \})$} 
\edge node[auto=right]{$/b$};[.{$(a, \{a,\emptyset\})$} \edge node[auto=right]{$/a$}; {$(\emptyset,\{\emptyset\})$} \edge node[auto=left]{$\setminus a$}; {$(\emptyset,\{\emptyset\})$} ] \edge node[auto=left]{$\setminus b$};[.{$(a, \{a,\emptyset\})$} \edge node[auto=right]{$/a$}; {$(\emptyset,\{\emptyset\})$} \edge node[auto=left]{$\setminus a$}; {$(\emptyset,\{\emptyset\})$} ] ]
\edge node[auto=left]{$\setminus c$};[.{$(ab,\{ab,b,\emptyset\})$} \edge node[auto=right]{$/b$}; [.{$(a,\{a,\emptyset\})$} \edge node[auto=right]{$/a$}; {$(\emptyset,\{\emptyset\})$} \edge node[auto=left]{$\setminus a$}; {$(\emptyset,\{\emptyset\})$} ] \edge node[auto=left]{$\setminus b$}; {$(a,\{\emptyset\})$} ] ]
\end{tikzpicture}
\caption{The computation of $\mathcal{T}(D)$ for the delta-matroid $D$ with ground-set $\{a,b,c\}$ under the total order $a < b < c$ and collection of feasible sets $\mathcal{P}(\{a,b,c\}) \setminus \{a\}$. Note that this delta-matroid is not ribbon-graphic (nor belongs to the larger class of ``vf-safe'' delta-matroids) \cite{2016arXiv160201306C}. For convenience, we eliminate brackets and commas in set notation wherever possible, so that, for example, $(ab,\{ab,b,\emptyset\})$ is the delta-matroid with ground-set $\{a,b\}$ and collection of feasible sets $\{ \{a,b\}, \{b\}, \emptyset\}.$ }
\end{figure}

Denote by $\mathcal{L}(D)$ the set of leaves of $\mathcal{T}(D)$. Let $D'$ be a node of $\mathcal{T}(D)$. Let $P$ be the unique path from $D$ to $D'$ in $\mathcal{T}(D)$. Each edge of $P$ corresponds to either deleting a point or contracting a point. Let $D'_c$ be the set of points contracted when following $P$ and let $D'_d$ be the set of points deleted when following $P$. If $D'$ is a leaf, then every point of $D'$ is either a loop or a coloop. Let $D'_{\text{co}}$ be the set of coloops of $D'$ and $D'_{\text{lo}}$ be the set of loops of $D'$.

Note that we obtain the following expression for $Q_{(w,x,0)}(D;y)$ in terms of $\mathcal{L}(D)$.

\begin{lemma} \label{lem:leaves} Let $D = (E,\mathcal{F})$ be a delta-matroid. Let $<$ be a total order on $E$. Then
\begin{align}
    Q_{(w,x,0)}(D;y) &= \sum_{L \in \mathcal{L}(D)} w^{|L_d|} x^{|L_c|} (x + wy)^{|L_{\text{co}}|} (w + xy)^{|L_{\text{lo}}|}. \label{eq:eeee}
\end{align}
\end{lemma}

\begin{proof} This follows directly from Theorem \ref{thm:transdc}.
\end{proof}

In the remainder of this section, we will compute the exponents in Equation \ref{eq:eeee} in terms of activities with respect to feasible sets, thereby obtaining a feasible-set expansion of $Q_{(w,x,0)}(D;y)$. First, note that we can characterize feasible sets in nodes of $\mathcal{T}(D)$.

\begin{lemma} \label{lem:feasiblenodes} Let $D = (E,\mathcal{F})$ be a delta-matroid and let $<$ be a total order on $E$. Let $D'$ be a node of $\mathcal{T}(D)$. A set $F' \subseteq E(D')$ is feasible in $D'$ if and only if $F = F' \cup D'_c$ is feasible in $D$.
\end{lemma}

\begin{proof} Note that since only nonsingular points are deleted or contracted when forming $D'$, the definitions of deletion and contraction imply that
\begin{equation*}
    \mathcal{F}(D') = \{ F \setminus D'_c : F \in \mathcal{F}, D'_c \subseteq F, \text{ and } F \cap D'_d = \emptyset\}.
\end{equation*}
The result follows.
\end{proof}

Lemma \ref{lem:feasiblenodes} motivates the following definition, which we state for arbitrary sets not only for feasible sets.

\begin{definition}
 Let $D = (E,\mathcal{F})$ be a delta-matroid and let $<$ be a total order on $E$. Let $D'$ be a node of $\mathcal{T}(D)$. We say $A \subseteq E(D)$ is \emph{covered} by $D'$ if $D'_c \subseteq A \subseteq D'_c \cup E(D').$
\end{definition}

We can now show that each leaf covers a unique feasible set and each feasible set is covered by a unique leaf, allowing us to rewrite the sum of Equation \ref{eq:eeee} as a sum over feasible sets as follows in Lemma \ref{lem:ffff}.

\begin{lemma} \label{lem:leafcover} Let $D = (E,\mathcal{F})$ be a delta-matroid and let $<$ be a total order on $E$. Every subset of $E$ is covered by exactly one leaf of $\mathcal{T}(D)$.
\end{lemma}

\begin{proof} Let $A \subseteq E$. Then $A$ corresponds to the following unique maximal path $P$ in $\mathcal{T}(D)$. The first node of $P$ is $D$. Suppose we have constructed the first $k$ nodes $D_1,D_2,\ldots,D_k$ of $P$. If $D_k$ is a leaf of $\mathcal{T}(D)$ we are done. Otherwise, let $a$ be  the highest nonsingular point of $D_k$. If $a \in A$, add $D_k / a$ to $P$. Otherwise, add $D_k \setminus a$ to $P$.

Let $L(A)$ be the final node of $P$. By construction, $L(A)_c \subseteq A$ and any points in $A \setminus L(A)_c$ are in $E(L(A))$. Thus, $L(A)$ covers $A$. Suppose some other leaf $L'(A)$ covers $A$. Let $P'$ be the unique path from $D$ to $L'(A)$. Let $D'$ be the highest-level node in $P \cap P'$. Since $L'(A)$ is distinct from $L(A)$, $D'$ is not a leaf. So let $b$ be the highest-ordered nonsingular point of $D'$. The next node of $P$ is $D' / b$ if and only if the next node of $P'$ is $D' \setminus b$. That is, $b \in L(A)_c$ if and only if $b \in E \setminus (L'(A)_c \cup E(L'(A)))$. Thus, $L'(A)$ cannot cover $A$.
\end{proof}

\begin{lemma} \label{lem:onefeasible} Let $D = (E,\mathcal{F})$ be a delta-matroid and let $<$ be a total order on $E$. Let $L$ be a leaf of $\mathcal{T}(D)$. Then there is a unique feasible set $F = L_{\text{co}} \cup L_c$ covered by $L$.
\end{lemma}

\begin{proof} Since every point of $L$ is singular, $L$ has a unique feasible set consisting of all its coloops, namely $L_{\text{co}}$. Therefore, by Lemma \ref{lem:feasiblenodes}, $L$ covers a unique feasible set $F = L_{\text{co}} \cup L_c$ of $D$.
\end{proof}

Lemma \ref{lem:onefeasible} allows us to rewrite the expansion in Equation \ref{eq:eeee} of $Q_{(w,x,0)}(D;y)$ over leaves of $\mathcal{T}(D)$ as a feasible-set expansion. Let $D = (E,\mathcal{F})$ be a delta-matroid and let $<$ be a total order on $E$. For each $F \in \mathcal{F}$, let $L(F)$ be the unique leaf covering $F$ guaranteed by Lemma \ref{lem:leafcover}.

\begin{lemma} \label{lem:ffff} Let $D = (E,\mathcal{F})$ be a delta-matroid and let $<$ be a total order on $E$. Then \begin{align}
    Q_{(w,x,0)}(D;y) &= \sum_{F \in \mathcal{F}}  w^{|L(F)_d|} x^{|L(F)_c|} (x + wy)^{|L(F)_{\text{co}}|} (w + xy)^{|L(F)_{\text{loop}}|}. \label{eq:ffff}
\end{align}
\end{lemma}

By computing the exponents of Equation \ref{eq:ffff} in terms of activities with respect to $F$, we will obtain our desired activity-based feasible-set expansion. To that end, we show in Theorem \ref{thm:activeorientable} that $E(L(F))$ is precisely the set of active and orientable points with respect to $F$. Note that this result generalizes theorems for quasi-trees in ribbon-graphs appearing in \cite{CKS2011,VT2011,B2012,De}. We will first need a lemma regarding interlacement and orientability.

\begin{lemma} \label{prop:trivialorientable} Let $D = (E,\F)$ be a delta-matroid and let $F \in \F$ and $a \in E$. If $a$ is $F$-orientable, then any feasible set in $D*F$ containing $a$ intersects $I(F;a)$.
\end{lemma}

\begin{proof} Suppose $a$ is $F$-orientable. Let $H$ be a feasible set in $D*F$ containing $a$. Since $\emptyset$ is feasible in $D*F$, we can apply the Symmetric Exchange Axiom to $\emptyset \Delta H$. Since $a \in \emptyset \Delta H$, there exists $b \in \emptyset \Delta H$ such that $\emptyset \Delta \{a,b\} = \{a,b\}$ is feasible in $D*F$. Since $a$ is $F$-orientable, $a \neq b$. Thus, $a$ and $b$ form an $F$-pair in which at most one point is nonorientable, i.e. $a$ and $b$ are $F$-interlaced. Thus, $b \in H \cap I(F;a)$.
\end{proof}

Nonsingular points are essential to constructing $\mathcal{T}(D)$, and so it will be useful to be able to recognize these in terms of the covering relation.

\begin{lemma} \label{lem:feasnonsingular} Let $D = (E,\mathcal{F})$ be a delta-matroid and let $<$ be a total order on $E$. Let $D'$ be a node of $\mathcal{T}(D)$. Let $a \in E(D')$. Suppose there exist distinct feasible sets $F$ and $F'$ of $D$ both covered by $D'$, with $a \in F$ and $a \not \in F'$. Then $a$ is nonsingular in $D'$.
\end{lemma}

\begin{proof} Since $F$ and $F'$ are covered by $D'$, $F = D'_c \cup H$ and $F_2 = D'_c \cup H'$ where $H, H' \in \mathcal{F}(D')$. Now $a \not \in D'_c$, and so $a \in H$. Moreover, $a \not \in F'$ and therefore $a \not \in H'$. Thus, $a$ is in at least one feasible set of $D'$, and so is not a loop, and $a$ is not in at least one feasible set of $d'$, and so is not a coloop. Therefore, $a$ is nonsingular.
\end{proof}

\begin{lemma}\label{lem:nonsingular} Let $D = (E,\mathcal{F})$ be a delta-matroid and let $<$ be a total order on $E$. Let $D'$ be a node of $\mathcal{T}(D)$ covering the feasible set $F$. Suppose $a,b \in E(D')$ are $F$-interlaced, and at least one of $a$ or $b$ is  $F$-orientable. Then $a$ and $b$ are both nonsingular in $D'$.
\end{lemma}

\begin{proof} Since $a$ and $b$ are $F$-interlaced and at least one is $F$-orientable, $F' = F \Delta \{a,b\}$ is feasible in $D$. Since $F$ is covered by $D'$ and both $a$ and $b$ are in $E(D')$, $F'$ is also covered by $D'$. Now $a \in F$ if and only if $a \not \in F'$ and $b \in F$ if and only if $b \not \in F'$. Thus, by Lemma \ref{lem:feasnonsingular}, both $a$ and $b$ are nonsingular in $D'$.
\end{proof}

We can now prove that $E(L(F))$ is the set of active and orientable points with respect to $F$.

\begin{theorem} \label{thm:activeorientable}Let $D = (E,\mathcal{F})$ be a delta-matroid and let $<$ be a total order on $E$. Let $F \in \mathcal{F}$. Let $L = L(F)$. Then $a \in E(L)$ if and only if $a$ is active and orientable with respect to $F$.
\end{theorem}

\begin{proof} Suppose $a \in E(L)$. By way of contradiction, suppose $a$ is $F$-nonorientable. Then there exists $F' \in \F$ such that $F \Delta F' = \{a\}$. But then $F \neq F'$ and $L$ covers $F'$. This contradicts Lemma \ref{lem:onefeasible}. Thus, $a$ is orientable. Next, we will show that $a$ is active. Suppose $b$ is $F$-interlaced with $a$. Since $a$ is orientable, $F' = F \Delta \{a,b\} \in \mathcal{F}$. Observe that if $b$ is also in $E(L)$, then $L$ would cover $F'$, contradicting Lemma \ref{lem:onefeasible}. Thus, $b \not \in E(L)$. Hence, there exists a unique nearest ancestor $D'$ of $L$ in $\mathcal{T}(D)$ such that $b \in E(D')$ and $b$ is the highest-ordered nonsingular point of $D'$. Now $F$ is covered by $D'$, and we have that $a,b \in E(D')$ are $F$-interlaced with $a$ being $F$-orientable. Thus, by Lemma \ref{lem:nonsingular}, $a$ is nonsingular in $D'$, implying $b>a$.

Now, suppose $a$ is active and orientable with respect to $F$. By way of contradiction, suppose $a \not \in E(L)$. Let $D'$ be the nearest ancestor of $L$ such that $a \in E(D')$ and $a$ is the highest-ordered nonsingular point of $D'$. We claim that $I(F;a) \cap E(D') = \emptyset$. Indeed suppose otherwise, and let $b \in I(F;a) \cap E(D')$. Then $a,b \in E(D')$ are $F$-interlaced, $F$ is covered by $D'$, and $a$ is $F$-orientable. Thus, by Lemma \ref{lem:nonsingular}, $b$ is nonsingular in $D'$. But since $a$ is active, $a < b$, which contradicts the assumption that $a$ is the highest-ordered nonsingular point of $D'$. Thus, $b \not \in E(D')$. Therefore, $I(F;a) \cap E(D') = \emptyset.$

Consider the following (mutually exclusive) cases: either $D' / a$ covers $F$ or $D' \setminus a$ covers $F$. Suppose $D' \setminus a$ covers $F$. Since $a$ is nonsingular in $D'$, there is a feasible set $F'$ covered by $D / a$. Note that since $F$ is covered by $D' \setminus a$ and $F'$ is covered by $F /a$, $a \in F \Delta F'$. Moreover, by the definition of twisting, $F \Delta F'$ is feasible in $D*F$. Thus, by Lemma \ref{prop:trivialorientable}, there exists some point $b \in I(F;a) \cap (F \Delta F').$ Since $I(F;a) \cap E(D') = \emptyset$, $b \not \in E(D').$ Thus, $b \in D'_c$ or $b \in D'_d$. Since both $F$ and $F'$ are covered by $D$, if $b \in D'_c$ then $b \in F \cap F'$ and if $b \in D'_d$ then $b \not \in F \cup F'$. In either case, we find that $b \not \in F \Delta F'$, a contradiction.

The case where $D' / a$ covers $F$ leads to a similar contradiction. Therefore, $a \in E(L)$ as desired.
\end{proof}

\begin{cor} \label{cor:core} Let $D = (E,\mathcal{F})$ be a delta-matroid and let $<$ be a total order on $E$. Let $F \in \mathcal{F}$. Let $\text{In}(F)$ be the set of internal, active, and orientable points with respect to $F$ and let $\text{Ex}(F)$ be the set of external, active, and orientable points with respect to $F$. Set $i(F) = |\text{In}(F)|$ and $j(F) = |\text{Ex}(F)|.$ Then
\begin{enumerate}
    \item $|L(F)_{\text{co}}| = i(F)$,
    \item $|L(F)_{\text{lo}}| = j(F)$
    \item $|L(F)_c| + |L(F)_{\text{co}}| = |F|$, and
    \item $|L(F)_d| = |E| - |F| - j(F)$.
\end{enumerate}
\end{cor}

\begin{proof} First, note that Theorem \ref{thm:activeorientable} implies the equalities
\begin{equation*}
    \text{In}(F) = E(L(F)) \cap F = L(F)_{\text{co}}
\end{equation*}
and
\begin{equation*}
    \text{Ex}(F) = E(L(F)) \setminus F = L(F)_{\text{lo}}
\end{equation*}
of sets. Part (1) follows from the first of these equalities, while Part (2) follows from the second. 

Next, by Lemma \ref{lem:onefeasible} we have $F = L(F)_c \cup L(F)_{\text{co}}$, where the union is disjoint since, by definition, $L(F)_{\text{co}} \subseteq E(L(F))$ while $L(F)_c \cap E(L(F)) = \emptyset.$ Thus, $|F| = |L(F)_c| + |L(F)_{\text{co}}$, proving part (3).

To prove Part (4), note that the points deleted from $D$ to obtain $L$ are precisely those points that are neither in $E(L)$ nor were contracted from $D$ to obtain $L$. That is, \begin{equation*}L(F)_d = E \setminus (E(L(F)) \cup L(F)_c).\end{equation*} But $E(L(F)) \cup L(F)_c = F \sqcup L(F)_{\text{lo}}$, where $\sqcup$ denotes disjoint union, and so \begin{equation*}|L(F)_d| = |E \setminus (F \sqcup L(F)_{\text{lo}})| = |E| - |F| - j(F).\end{equation*}
\end{proof}

We can now compute a feasible-set expansion for the transition polynomial $Q_{(w,x,0)}(D;y)$.

\begin{theorem} \label{thm:transfeas} Let $D = (E,\mathcal{F})$ be a delta-matroid. Let $<$ be any total order on $E$. Let $F \in \mathcal{F}$. Let $i(F)$ be the number of internal, active, and orientable points with respect to $F$ and let $j(F)$ be the number of external, active, and orientable points with respect to $F$. Then
\begin{equation}
    Q_{(w,x,0)}(D;y) = \sum_{F \in \mathcal{F}} w^{|E| - |F|} x^{|F|} (1 + (w/x)y)^{i(F)} (1 + (x/w)y)^{j(F)}
\end{equation}
\end{theorem}

\begin{proof} Note that
\begin{align}
    Q_{(w,x,0)}(D;y) &= \sum_{F \in \mathcal{F}} w^{|L(F)_d|} x^{|L(F)_c|} (x + wy)^{|L(F)_{\text{co}}|} (w + xy)^{|L(F)_{\text{lo}}|} \notag \\
    &= \sum_{F \in \mathcal{F}} w^{|L(F)_d|} x^{|L(F)_c| + |L(F)_{\text{co}}|} (1 + (w/x)y)^{|L(F)_{\text{co}}|} (w + xy)^{|L(F)_{\text{lo}}|}\notag \\
    &= \sum_{F \in \mathcal{F}} w^{|E| - |F| - j(F)} x^{|F|} (1 + (w/x)y)^{i(F)} (w + xy)^{j(F)} \notag \\
    &= \sum_{F \in \mathcal{F}} w^{|E| - |F|} x^{|F|} (1 + (w/x)y)^{i(F)} (1 + (x/w)y)^{j(F)} \notag
\end{align}
where the first equality is simply that of Equation \ref{eq:ffff}, the third follows from Corollary \ref{cor:core}, and the remainder are standard algebraic manipulations.
\end{proof}

\section{The Interlace and Bollob\'{a}s-Riordan Polynomials}

Delta-matroids associated to ribbon-graphs are defined in terms of the edge-set of the ribbon graph, while delta-matroids associated to abstract graphs are defined in terms of the vertex-set of the abstract graph. As a result, both graph polynomials defined in terms of edge-sets (like the Bollob\'{a}s-Riordan polynomial) and graph polynomials defined in terms of vertex-sets (like the interlace polynomial) have been generalized to delta-matroids \cite{CMNR2014,BH2014}. In fact, both the Bollob\'{a}s-Riordan polynomial and interlace polynomial are specializations of the transition polynomial $Q_{(w,x,0)}(D;y)$. Thus, applying our main result, we can compute feasible set expansions of each of these polynomials in terms of delta-matroid activities. We give in Figure \ref{fig:diagram} a diagram of the combinatorial objects and polynomials involved in this section.

\begin{figure}
    \centering
    \begin{tikzcd}
    & &  q(G) \arrow[dr,dotted] \\
        \text{graphs} \arrow[rd,dotted] \arrow[rrd,bend left,dotted,"adjacency"] \arrow[urr,rightarrow,bend left] &  & & q(D)  \\
        & \Delta-\text{matroids} \arrow[r,rightarrow] & \text{interlacement} \arrow[r,Rightarrow] & Q(D) \arrow[u,Rightarrow] \arrow[d,Rightarrow] &   \\
        \text{ribbon-graphs} \arrow[ru,dashed] \arrow[rru,bend right,dashed,"interlacement"] \arrow[drr,bend right,rightarrow] & & & \tilde{R}(D) \arrow[dd, bend left, Rightarrow,"matroid"] \\
        & & \tilde{R}(G) \arrow[ur,dashed] \arrow[dr,bend right,rightarrow,"plane"]\\
        &&& T
    \end{tikzcd} 
    \caption{A diagram of the combinatorial objects and polynomials we consider.  Dotted arrows indicate the map from abstract graphs to delta-matroids. Dashed arrows indicate the map from embedded graphs to delta-matroids. Doubled arrows should be read as ``provides a feasible-set expansion for.'' The polynomials are named without reference to variables -- $q$ is the two-variable interlace polynomial of a graph or delta-matroid, $Q$ is the transition polynomial $Q_{(w,x,0)}(D;y)$, $\tilde{R}$ is the two-variable Bollob\'{a}s-Riordan polynomial of a ribbon graph or delta-matroid, and $T$ is the Tutte polynomial of a graph or matroid.}
    \label{fig:diagram}
\end{figure}
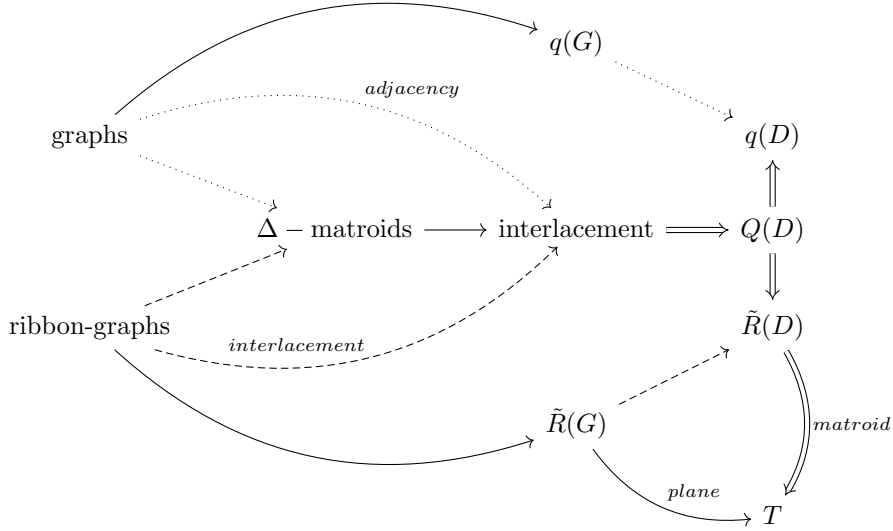

\subsection{The Interlace Polynomial}

Brijder and Hoogeboom recently generalized \cite{BH2014} both the two-variable and one-variable interlace polynomials to the setting of delta-matroids.

\begin{definition} \cite{BH2014} \label{def:interlace} Let $D = (E,\mathcal{F})$ be a delta-matroid. The \emph{two-variable interlace polynomial} of $D$ is 
\begin{equation}
    \bar{q}(D;x,y) := Q_{(1,x,0)}(D;y) = \sum_{A \subseteq E} x^{|A|} y^{r(D*A)}.
\end{equation}
The \emph{one-variable interlace polynomial} is
\begin{equation}
    q_{\Delta}(D;y) := \bar{q}(D;1,y) = \sum_{A \subseteq E} y^{r(D*A)}.
\end{equation}
\end{definition}

Theorem \ref{thm:transfeas} gives the following feasible-set expansion of $\bar{q}(D;x,y)$.

\begin{theorem} \label{thm:expand} Let $D = (E,\mathcal{F})$ be a delta-matroid. Let $<$ be any total order on $E$. Let $F \in \mathcal{F}$. Let $i(F)$ be the number of internal, active, and orientable points with respect to $F$ and let $j(F)$ be the number of external, active, and orientable points with respect to $F$. Then
\begin{equation}
    \bar{q}(D;x,y) = \sum_{F \in \mathcal{F}} x^{|F|} (1 + y/x)^{i(F)} (1 + xy)^{j(F)}.
\end{equation}
\end{theorem}

\begin{proof} This is an immediate consequence of Theorem \ref{thm:transfeas}.
\end{proof}

As a corollary, we obtain the following feasible-set expansion of the single-variable interlace polynomial.

\begin{cor} Let $D = (E,\mathcal{F})$ be a delta-matroid. For any total order $<$ on $E$ we have
\begin{align*}
   q_\Delta(D;y-1) = \sum_{F \in \mathcal{F}} y^{i(F) + j(F)}.
\end{align*}
\end{cor}

\begin{proof} We have
\begin{align}
    q_{\Delta}(D;y-1) = \bar{q}(D;1,y-1) = \sum_{F \in \mathcal{F}} y^{i(F) + j(F)} \notag.
\end{align}
\end{proof}

\subsection{The Bollob\'{a}s-Riordan Polynomial}

The Bollob\'{a}s-Riordan polynomial was originally constructed as a generalization of the Tutte polynomial to the domain of ribbon graphs. Recently, Chun et al. have shown that the Bollob\'{a}s-Riordan polynomial of a ribbon graph is determined by its ribbon-graphic delta-matroid, and so they obtained a generalization of the Bollob\'{a}s-Riordan polynomial to delta-matroids \cite{2016arXiv160201306C}. A normalized two-variable version of this polynomial (originally studied for ribbon graphs) has taken on additional significance for delta-matroids due to recent results of Krajewski et al. on combinatorial Hopf algebras \cite{2015arXiv150800814K}. In particular, they have shown that the two-variable Bollob\'{a}s-Riordan polynomial of a delta-matroid is in some sense the canonical Tutte-like polynomial for delta-matroids under usual deletion and contraction \cite{2015arXiv150800814K}. In this section, we use a connection between the two-variable interlace polynomial and this two-variable Bollob\'{a}s-Riordan polynomial to obtain a feasible-set expansion of the latter in terms of delta-matroid activities. 

The two-variable Bollob\'{a}s-Riordan polynomial is defined in terms of the following rank function.

\begin{definition} (rank functions) \cite{CMNR2014} Let $D = (E,\mathcal{F})$ be a delta-matroid.  Define $\sigma(D) = \frac{1}{2}( r_{\max}(E) + r_{\min}(E))$. For $A \subseteq E$, define $\sigma_D(A) = \sigma(D|A)$. Recall that $w(D) = r_{\max}(E) - r_{\min}(E)$.
\end{definition}

\begin{definition} (Bollob\'{a}s-Riordan polynomial) \cite{CMNR2014} Let $D = (E,\mathcal{F})$ be a delta-matroid. The \emph{Bollob\'{a}s-Riordan polynomial} of $D$ is
\begin{equation*}
    \tilde{R}(D;x,y) = \sum_{A \subseteq E} (x-1)^{\sigma(E) - \sigma(A)} (y-1)^{|A| - \sigma(A)}.
\end{equation*}
\end{definition}

Note that if $D$ is a matroid, $\sigma_D$ is  precisely the rank function of $D$. Hence, for matroids, the Bollob\'{a}s-Riordan polynomial is the usual Tutte polynomial. The Bollob\'{a}s-Riordan polynomial can be obtained from a three-variable version.

\begin{definition} (three-variable Bollob\'{a}s-Riordan polynomial) \cite{CMNR2014} Let $D = (E,\mathcal{F})$ be a delta-matroid. The \emph{three-variable Bollob\'{a}s-Riordan polynomial} of $D$ is
\begin{equation*}
    R(D;x,y,z) = \sum_{A \subseteq E} (x-1)^{r_{\min}(E) - r_{\min}(A)} y^{|A| - r_{\min}(A)} z^{w(D|A)}.
\end{equation*}
\end{definition}

\begin{theorem} \cite{CMNR2014} \label{thm:threetotwo} Let $D = (E,\mathcal{F})$ be a delta-matroid. Then \begin{equation*}
    \tilde{R}(D;x+1,y+1) = x^{w(D)/2} R(D;x+1,y,1/\sqrt{xy})
\end{equation*}
\end{theorem}

The three-variable Bollob\'{a}s-Riordan polynomial can be obtained from the two-variable interlace polynomial (and transition polynomial).

\begin{theorem} \label{thm:polyrel} Let $D = (E,\mathcal{F})$ be a delta-matroid. Then
\begin{align}
    \bar{q}(D;\sqrt{y/x},\sqrt{xy}) = Q_{(1,\sqrt{y/x},0)}(D;\sqrt{xy}) = (\sqrt{y/x})^{r_{\min}(E)} R(D;x+1,y,1/\sqrt{xy}) \notag
\end{align}
\end{theorem}

\begin{proof} The first equality follows from the definition of the two-variable interlace polynomial, the second equality is Proposition 5.9 of \cite{2016arXiv160201306C}.
\end{proof}

We therefore obtain the following feasible-set expansion of the Bollob\'{a}s-Riordan polynomial.

\begin{theorem} Let $D = (E,\mathcal{F})$ be a delta-matroid. For $F \in \mathcal{F}$, let $i(F)$ be the number of internal, active, and orientable points with respect to $F$ and let $j(F)$ be the number of external, active, and orientable points with respect to $F$. Then
\begin{equation*}
    \tilde{R}(D;x,y) = \sum_{F \in \mathcal{F}} \left(x-1\right)^{(r_{\max}(E) - |F|)/2} \left(y-1\right)^{(|F| - r_{\min}(E))/2}  x^{i(F)} y^{j(F)} \notag
\end{equation*}

\end{theorem}

\begin{proof} Observe that
\begin{align}
    \tilde{R}(D;x+1,y+1) &= x^{w(D)/2} R(D;x+1,y,1/\sqrt{xy}) \notag \\
    &= \left( \frac{x^{w(D)/2}}{\sqrt{y/x}^{r_{\min}(E)}}\right) \sqrt{y/x}^{r_{\min}(E)} R(D;x+1,y,1/\sqrt{xy}) \notag \\
    &= \left(\sqrt{x^{r_{\max}(E)}/y^{r_{\min}(E)}}\right) Q_{(1,\sqrt{y/x},0)} (D; \sqrt{xy}) \notag \\
    &= \left( \sqrt{x^{r_{\max}(E)}/y^{r_{\min}(E)}} \right) \sum_{F \in \mathcal{F}} \left( \sqrt{y/x} \right)^{|F|} \left(1 + x \right)^{i(F)} \left(1 + y \right)^{j(F)} \notag \\
    &= \sum_{F \in \mathcal{F}} x^{(r_{\max}(E) - |F|)/2} y^{(|F| - r_{\min}(E))/2}  (x+1)^{i(F)} (y + 1)^{j(F)}. \notag
\end{align}

where the first equality follows from Theorem \ref{thm:threetotwo}, the third by Theorem \ref{thm:polyrel}, and the fourth by Theorem \ref{thm:transfeas}. The result follows by substituting $x$ and $y$ for $x+1$ and $y+1$.
\end{proof}

For matroids, this reduces to the usual basis expansion of the Tutte polynomial.

\begin{cor} \label{cor:matroid} Let $D = (E,\mathcal{F})$ be a matroid. Then
\begin{equation}
    \tilde{R}(D;x,y) = \sum_{F \in \mathcal{F}} x^{i(F)} y^{j(F)}.
\end{equation}
\end{cor}

\begin{proof} Since $D$ is a matroid, $r_{\max}(E) - |F| = |F| - r_{\min}(E) = 0$ for any $F \in \mathcal{F}$. Thus,
 \begin{align}
     \tilde{R}(M;x,y) &= \sum_{F \in \mathcal{F}} \left(x-1\right)^{(r_{\max}(E) - |F|)/2} \left(y-1\right)^{(|F| - r_{\min}(E))/2}  x^{i(F)} y^{j(F)} \notag \\
     &= \sum_{F \in \mathcal{F}} x^{i(F)} y^{j(F)}. \notag
 \end{align}
\end{proof}

\section{Conclusion}

We have shown that the transition polynomial $Q_{(w,x,0)}(D;y)$ has an activities based feasible-set expansion, and used this expansion to obtain activity expansions for the Bollob\'{a}s-Riordan and interlace polynomials. There are a number of open questions remaining.

For example, our result for the Bollob\'{a}s-Riordan polynomial applies only to the two-variable version. In \cite{CMNR2014}, the full Bollob\'{a}s-Riordan polynomial of a delta-matroid $D = (E,\mathcal{F})$ is given as a sum over subsets of $E$. Lemmas 4.7 and 4.9 partition the powerset of $E$ into parts each containing a single feasible set, and hence this sum over edge-sets can, in principle, be rewritten as a sum over feasible sets. Indeed, this was the strategy taken in \cite{CKS2011,De,VT2011} to compute spanning quasi-tree expansions of the Bollob\'{a}s-Riordan polynomial for ribbon graphs. It may be possible to take a similar approach to compute a feasible-set expansion of the full Bollob\'{a}s-Riordan polynomial for delta-matroids.

In general, this paper provides additional evidence that minor-based recursive definitions should correspond to activity-based expansions. However, we know of no general theoretical explanation of this connection. Multimatroids, objects generalizing matroids and delta-matroids, have a rich structure of minor operations: a $k$-matroid has in some sense $k$ different ``directions'' in which to take minors. Perhaps multimatroid theory can provide some insight into the connection between minor-based recursion and activities.

\section{Acknowledgements} Funding for this work was provided by the Vermont Space Grant Consortium. In addition, this research benefited greatly from the insights, at various stages, of Jo Ellis-Monaghan, Robert Brijder, Iain Moffatt, and Steven Noble. TikZ code for generating Figure \ref{fig:ribbonchar} was obtained from code written by Manuel B\"{a}renz.

\bibliographystyle{plain}
\bibliography{main}

\end{document}